\documentclass[a4paper,12pt,reqno]{amsart}
\usepackage{setspace} 
\usepackage{amsmath, amsfonts,amsthm,amssymb,mathrsfs,cases}
\usepackage{enumerate}
\usepackage[dvipdfmx]{graphicx}
\usepackage{ascmac}
\usepackage[arrow,matrix]{xy}
\usepackage{color}
\usepackage{array}
\usepackage{pifont}
\usepackage{longtable}
\usepackage{mathtools}

\theoremstyle{definition}
\newtheorem{theorem}{Theorem}[section]
\newtheorem*{theorem*}{Theorem}
\newtheorem{definition}[theorem]{Definition}
\newtheorem*{definition*}{Definition}
\newtheorem{lemma}[theorem]{Lemma}
\newtheorem*{lemma*}{Lemma}

\newtheorem*{example*}{Example}
\newtheorem{proposition}[theorem]{Proposition}
\newtheorem*{proposition*}{Proposition}

\newtheorem*{corollary*}{Corollary}

\newtheorem{maintheorem}{Theorem}

\newcommand{\ctext}[1]{\raise0.2ex\hbox{\textcircled{\scriptsize{#1}}}}

\DeclareMathOperator{\Spin}{Spin}
\DeclareMathOperator{\pr}{pr}
\DeclareMathOperator{\Id}{Id}
\DeclareMathOperator{\tr}{tr}
\DeclareMathOperator{\Ric}{Ric}
\DeclareMathOperator{\scal}{scal}
\DeclareMathOperator{\vol}{vol}

\DeclareMathOperator{\G_2}{G_2}

\DeclareMathOperator{\SU}{SU}

\DeclareMathOperator{\U}{U}

\DeclareMathOperator{\Sp}{Sp}
\DeclareMathOperator{\End}{End}

\DeclareMathOperator{\Sym}{Sym}

\def\C{\mathbb{C}}
\def\R{\mathbb{R}}
\def\P{\mathbb{P}}
\def\F{\mathbb{F}}
\title{Rarita-Schwinger fields on nearly K\"{a}hler manifolds}

\author{Soma Ohno and Takuma Tomihisa}

\address{Soma Ohno, Department of Pure and applied Mathematics, Graduate school of fundamental science
and engineering, Waseda University, 3-4-1 Ohkubo, Shinjuku-ku, Tokyo, 169-8555, Japan.}
\email{runhorse@fuji.waseda.jp}

\address{Takuma Tomihisa, Department of Applied Mathematics, Faculty of science and engineering, Waseda University, 3-4-1 Ohkubo, Shinjuku-ku, Tokyo, 169-8555, Japan.}
\email{taku-tomihisa@akane.waseda.jp}

\begin{document}

\begin{abstract}
We study Rarita-Schwinger fields on 6-dimensional compact strict nearly K\"{a}hler manifolds. In order to investigate them, we clarify the relationship between some differential operators for the Hermitian connection and the Levi-Civita connection. As a result, we show that the space of the Rarita-Schwinger fields coincides with the space of the harmonic 3-forms. Applying the same technique to a deformation theory, we also find that the space of the infinitesimal deformations of Killing spinors coincides with the direct sum of a certain eigenspace of the Laplace operator and the space of the Killing spinors.
\end{abstract}

\maketitle

\makeatletter
  \renewcommand{\theequation}{
  \thesection.\arabic{equation}}
 \@addtoreset{equation}{section}
\makeatother

\section{Introduction}
Rarita-Schwinger fields are solutions of the Rarita-Schwinger equation. This equation is the relativistic field equation of spin-3/2 fermions, introduced by W. Rarita and J. Schwinger \cite{RaritaSchwinger}. 
We can consider Rarita-Schwinger fields as sections in the kernel of the Rarita-Schwinger operator, a generalization of the classical Dirac operator for spin-1/2 fields, with divergence free. In physics, especially in relativity and superstring theory, Rarita-Schwinger fields play an important role because, for example, they describe the gravitino, the supersymmetric partner of the hypothesized graviton. On the other hand, in mathematics, there have been some studies about Rarita-Schwinger fields: Wang \cite{Wang} examined the relationship between these fields and infinitesimal Einstein deformations with Killing spinors. The Rarita-Schwinger operator is important in association with elliptic genus (cf. \cite{Witten}). The spectrum of the Rarita-Schwinger operator on some concrete symmetric spaces was computed by Homma and Tomihisa \cite{YasushiTomihisa}. B\"{a}r and Mazzeo \cite{BarMazzeo} proved that there exists a compact manifold with many Rarita-Schwinger fields in any given dimension $n \geq 4$. In particular, if $n$ is divisible by 4, then we can take this manifold as a simply connected compact manifold with negative Einstein constant. Homma and Semmelmann \cite{YasushiSemmelmann} classified the manifolds with parallel Rarita-Schwinger fields. They also find some Einstein manifolds with positive Einstein constant admitting Rarita-Schwinger fields, quaternionic K\"{a}hler manifolds, symmetric spaces, and some algebraic manifolds. Then, it is interesting to find an Einstein manifold with positive Einstein constant which has Rarita-Schwinger fields.

In this paper, we investigate whether there exist (non-trivial) Rarita-Schwinger fields on nearly K\"{a}hler 6-manifolds, which have positive Einstein constant. This solves one of the open questions which Homma and Semmelmann \cite{YasushiSemmelmann} proposed. Nearly K\"{a}hler manifolds were first studied by A. Gray \cite{Gray2}. A nearly K\"{a}hler manifold is characterized as an almost Hermitian manifold $(M,g,J)$ with $(\nabla_X J)X=0$ for any vector field $X$, where $\nabla$ is the Levi-Civita connection. The 6-dimensional case stands out because of the lowest dimension in which non-K\"{a}hler nearly K\"{a}hler manifolds appear, and because of the existence of real Killing spinors. Thus, 6-dimensional non-K\"{a}hler nearly K\"{a}hler manifolds are Einstein manifolds with positive scalar curvature. Furthermore, in this case, they have $\SU(3)$-structures.

We have three major ideas to study Rarita-Schwinger fields on nearly K\"{a}hler manifolds. The first  in Section \ref{Preliminalies} is to decompose the tangent bundle, the spinor bundle $S_{1/2}$, and $S_{3/2}$ by using the Killing spinors and the $\SU(3)$-structures. Here, in Section \ref{Preliminalies}, we also define Rarita-Schwinger fields and nearly K\"{a}hler manifolds. The second in Section \ref{differential operators} is to investigate relationships between various differential operators for the Levi-Civita connection and the Hermitian connection. The third in Section \ref{Rarita Schwinger fields on nearly Kahler} is to rewrite the Rarita-Schwinger equation into the simultaneous equation consisting of 2-forms and 3-forms by using these decompositions and relationships. As a result, we get the main result in Section \ref{Rarita Schwinger fields on nearly Kahler}. 

\begin{maintheorem} \label{thmA}
The space of the Rarita-Schwinger fields is isomorphic to the space of the harmonic 3-forms on 6-dimensional compact non-K\"{a}hler nearly K\"{a}hler manifolds.
\end{maintheorem}

In particular, in Section \ref{Examples and Applications}, we know that there is a two dimensional space of the Rarita-Schwinger fields on the nearly K\"{a}hler $S^3 \times S^3$. This is the first example of a six-dimensional Einstein manifold with positive scalar curvature admitting Rarita-Schwinger fields.

In the same way as we prove Theorem \ref{thmA}, we can obtain a theorem about deformation theory of Killing spinors. The general theory of deformations of Killing spinors was first developed by Wang \cite{Wang}. van Coevering \cite{Coevering} recently studied deformations of Killing spinors on 3-Sasakian manifolds. We study deformations of Killing spinors on nearly K\"{a}hler manifolds. Let $E(\lambda)$ be the $\lambda$-eigenspace of the Laplace operator restricted to co-closed primitive $(1, 1)$-forms, and $K_+$ be the space of the Killing spinors with the Killing number $\frac{1}{2}$. Then our result in Section \ref{Examples and Applications} is

\begin{maintheorem} \label{thmB}
Let $(M^6,g,J)$ be a 6-dimensional compact non-K\"{a}hler nearly K\"{a}hler manifold. Then the space of the infinitesimal deformations of Killing spinors is isomorphic to the space $E(12) \oplus K_+$.
\end{maintheorem}

There have been some works about deformation theory on nearly K\"{a}hler 6-manifolds (\cite{Foscolo}, \cite{MoroianuNagySemmelmann}, \cite{MoroianuSemmelmann}). These works have been conducted on 6-dimensional compact non-K\"{a}hler nearly K\"{a}hler manifolds. Moroianu, Nagy and Semmelmann \cite{MoroianuNagySemmelmann} showed that, except for the round sphere $S^6$, the space of the infinitesimal deformations of nearly K\"{a}hler structures is isomorphic to the space $E(12)$. As described in the paper \cite{Grunewald}, there is a one-to-one correspondence between the space of the Killing spinors (modulo constant rescaling) and the space of nearly K\"{a}hler structures. Since the space of the Killing spinors $K_+$ is one dimensional except for $S^6$, Theorem \ref{thmB} is a reproof of the result of \cite{MoroianuNagySemmelmann}. 

\section{Preliminalies} \label{Preliminalies}
\subsection{Rarita-Schwinger fields}
In this subsection, we introduce Rarita-Schwinger fields and some related operators. Details such as definitions and Weitzenb\"{o}ck formulas are listed in \cite{YasushiSemmelmann}, \cite{Wang}.

Let $(M,g)$ be an $n$-dimensional Riemannian spin manifold with the spinor bundle $S_{1/2}$ and the complexified tangent bundle $TM^{\C}$. We consider the twisted Dirac operator on $S_{1/2} \otimes TM^{\C}$,
\[
D_{TM} = \sum_{k=1}^n (e_k \cdot \otimes \Id_{TM^{\C}}) \circ \nabla_{e_k},
\]
where $\nabla$ is the covariant derivative on $S_{1/2} \otimes TM^{\C}$ induced from the Levi-Civita connection and $e_k \cdot$ is the Clifford multiplication by an orthonormal frame $\{ e_k \}$ of $TM$. We define the vector bundle $S_{3/2}$ as $\ker \Pi$ for the bundle map $\Pi : S_{1/2} \otimes TM^{\C} \ni \zeta \otimes X \mapsto X \cdot \zeta \in S_{1/2}$. With respect to the $\Spin(n)$ decomposition $S_{1/2} \otimes TM^{\C} \cong S_{1/2} \oplus S_{3/2}$, we can write $D_{TM}$ as the $2 \times 2$ matrix
\[
D_{TM} = \left(
 \begin{array}{ccc}
 \frac{2-n}{n} D & 2 P^{\ast} \\[1ex]
 \frac{2}{n}P & Q
 \end{array}
\right),
\]
where $D:\Gamma(S_{1/2}) \rightarrow \Gamma(S_{1/2})$ is the Dirac operator in the ordinary sense, $P:\Gamma(S_{3/2}) \rightarrow \Gamma(S_{1/2})$ is the Penrose operator, and $P^{\ast}$ is the formal adjoint operator of $P$. The operator $Q:\Gamma(S_{3/2}) \rightarrow \Gamma(S_{3/2})$ is called {\it the Rarita-Schwinger operator}, which is a formally self-adjoint elliptic differential operator of first order. 

{\it A Rarita-Schwinger field} is a section $\phi$ of $S_{1/2} \otimes TM^{\C}$ that satisfies
\[
\phi \in \Gamma(S_{3/2}), \; P^{\ast}\phi = 0, \; {\rm and} \; Q\phi = 0.
\]
These equations are equivalent to $\phi \in \Gamma(S_{3/2})$ and $D_{TM}\phi = 0$. The result of Homma and Semmelmann \cite[Proposition 4.1]{YasushiSemmelmann} tells us the following. If $(M^n,g)$, $n \geq 3$, is a compact Einstein spin manifold with non-negative scalar curvature, then we have $\ker Q \cap \ker P^{\ast} = \ker Q$.

\subsection{Nearly K\"{a}hler manifolds} \label{nearly Kahler}
An almost Hermitian manifold $(M^{2m},g,J)$ is called {\it a nearly K\"{a}hler manifold} if
\begin{equation} \label{nearly Kahler condition}
(\nabla_X J)X = 0 \quad {\rm for \: all} \; X \in \Gamma(TM),
\end{equation}
where $\nabla$ denotes the Levi-Civita connection of $g$. {\it The canonical Hermitian connection} $\bar{\nabla}$, defined by
\begin{equation} \label{herm conn TM}
\bar{\nabla}_{X}Y \coloneqq \nabla_X Y - \frac{1}{2}J(\nabla_X J)Y \quad {\rm for \: all} \; X,Y \in \Gamma(TM),
\end{equation}
satisfies $\bar{\nabla}g=0$ and $\bar{\nabla}J=0$. Note that the torsion of $\bar{\nabla}$ given by $\bar{T}(X,Y)=-J(\nabla_X J)Y$ vanishes iff $(M,g,J)$ is a K\"{a}hler manifold.

Nagy \cite{NagyAndy} proved that every compact simply connected nearly K\"{a}hler manifold is isometric to the product of several Riemannian manifolds. Each manifold is in one of the following classes of nearly K\"{a}hler manifold: K\"{a}hler manifolds, naturally reductive 3-symmetric spaces, twistor spaces over compact quaternion-K\"{a}hler manifolds with positive scalar curvature, and 6-dimensional nearly K\"{a}hler manifolds. From now on, we will consider only compact 6-dimensional strict (i.e. non-K\"{a}hler) nearly K\"{a}hler manifolds. It is an important fact that the Ricci curvature satisfies $\Ric=5g$. In particular, these manifolds are Einstein manifolds. Note that compactness and completeness are equivalent on 6-dimensional strict nearly K\"{a}hler manifolds, because the Ricci curvature is positive. 

We denote as usual the K\"{a}hler form of $M$ by $\omega \coloneqq g(J\cdot, \cdot)$. The tensor $\psi^+ \coloneqq \nabla \omega$ is a 3-form by $(\ref{nearly Kahler condition})$, and the tensor $\psi^- \coloneqq \ast \psi^+$ is also a 3-form, where $\ast$ is the Hodge star operator. The $\bar{\nabla}$-parallel complex volume form is represented as $\psi^+ + i\psi^-$. Also, we know the (real) volume form $\vol = \vol_g$ coincides with $\frac{1}{4} \psi^+ \wedge \psi^-$. It is a prominent property that the pair $(\omega, \psi^+)$ leads to an $\SU(3)$-structure on $TM$ (cf. \cite{Hitchin}). 

For a vector field $X$, let $A_X$ denote the section $A_X = J(\nabla_X J)$ of $\End TM$. By definition, the canonical Hermitian connection is written as $\bar{\nabla}_X = \nabla_X - \frac{1}{2}A_X$ on the tangent bundle. For every endomorphism or 2-tensor $B$ in $\End TM \cong T^{\ast}M \otimes T^{\ast}M$, we denote by $B_{\star}$ in $\End \mathcal{T}M$ the natural extension of $B$ (cf. \cite[p.3059]{MoroianuSemmelmann}), where $\mathcal{T}M$ is a tensor bundle. Then a similar equation $\bar{\nabla}_X = \nabla_X - \frac{1}{2}A_{X \star}$ holds on $\mathcal{T}M$. From now on, we identify $TM$ with $T^{\ast}M$ using the metric without notice.

As stated in \cite{Gray}, a spin structure exists on a 6-dimensional strict nearly K\"{a}hler manifold. The spin connection $\nabla$ is obtained by pulling back the Levi-Civita connection to the spinor bundle. Similarly, we pull back the canonical Hermitian connection to the spinor bundle and also denote this connection on the spinor bundle as $\bar{\nabla}$, which is written explicitly as
\begin{equation} \label{herm conn S}
\bar{\nabla}_X \zeta = \nabla_X \zeta - \frac{1}{4}X \mathbin{\lrcorner} \psi^- \cdot \zeta \quad {\rm for \: all} \; X \in \Gamma(TM), \zeta \in \Gamma(S_{1/2}).
\end{equation}

\subsection{Algebraic results on nearly K\"{a}hler manifolds} \label{Algebraic results}
In the subsections \ref{Algebraic results} and \ref{curvatureLaplacian}, we mainly summarize the contents of the articles \cite{Foscolo} and \cite{MoroianuSemmelmann} which we will use in later calculations. Assume that $(M^6,g,J)$ is a 6-dimensional strict nearly K\"{a}hler manifold with the normalized scalar curvature $\scal = 30$. We decompose the exterior bundles $\wedge^2 M$, $\wedge^3 M$, and the spinor bundle $S_{1/2}$. Also, we write the action of differential forms on the Killing spinors by the Clifford multiplication explicitly. 

The exterior bundle $\wedge^2 M$ decomposes into $\SU(3)$ irreducible components as follows:
\[
\wedge^2 M \cong \R \omega \oplus \wedge^{(2,0)+(0,2)} M \oplus \wedge^{(1,1)}_0 M,
\]
where $\wedge^{(1,1)}_0 M$ is the bundle of primitive $(1,1)$-forms. The map $X \mapsto X \mathbin{\lrcorner} \psi^+$ identifies the second summand with $TM$. The map $h \mapsto g(Jh \cdot, \cdot)$ defines an isomorphism between $\Sym^+_0 M$ and the third summand, where $\Sym^+ M$ is the bundle of symmetric endomorphisms commuting with $J$, and $\Sym^+_0 M$ is the trace-free part of $\Sym^+ M$. Similarly, one can decompose the exterior bundle $\wedge^3 M$ into $\SU(3)$ irreducible components
\[
\wedge^3 M \cong \wedge^1 M \wedge \omega \oplus \wedge^{(3,0)+(0,3)} M \oplus \wedge^{(2,1)+(1,2)}_0 M.
\]
The second summand is a rank 2 trivial bundle spanned by the forms $\psi^{\pm}$. The isomorphism $S \mapsto S_{\star}\psi^+$ identifies $\Sym^- M$, the bundle of symmetric endomorphisms anticommuting with $J$, with the third summand. Here we note that the bundle of symmetric endomorphisms $\Sym M$ is equal to $\Sym^+ M \oplus \Sym^- M$.

Since the scalar curvature is normalized as $\scal = 30$, $M$ admits a unit Killing spinor $\kappa$ with the Killing number $\frac{1}{2}$. The Killing spinor $\kappa$ defines a bundle map $\gamma \mapsto \gamma \cdot \kappa$ from $\wedge^0 M \oplus \wedge^1 M \oplus \wedge^6 M$ to $S_{1/2}$. This map preserves the inner product, so must be injective. Since its domain and target have equal rank, it is an isomorphism: 
\begin{equation} \label{irrep decomp S_{frac{1}{2}}}
\wedge^0 M \oplus \wedge^1 M \oplus \wedge^6 M \cong S_{1/2}.
\end{equation}

We give various actions on the Killing spinor, which is used in later calculations. From the construction of the Killing spinor (see \cite{Grunewald}), the equation is obtained: 
\begin{equation} \label{5-form}
JX \cdot \kappa = \vol \cdot X \cdot \kappa = -X \cdot \vol \cdot \kappa \quad {\rm for \: all} \; X \in \Gamma(TM).
\end{equation}
Next, the form $\psi^- \in \Omega^{(3,0)+(0,3)}M$ acts as follows (cf. \cite[Lemma 2]{CharbonneauHarland}):
\begin{equation} \label{3-form}
\psi^- \cdot \kappa = 4\kappa.
\end{equation}
For any 2-form $\eta= \lambda \omega + Y \mathfrak{\lrcorner} \psi^+ + \eta_0 \in C^{\infty}(M)\omega \oplus \Omega^{(2,0)+(0,2)}M \oplus \Omega^{(1,1)}_0 M$, we have 
\begin{equation} \label{2-form}
\eta \cdot \kappa = 3\lambda \vol \cdot \kappa + 2JY \cdot \kappa.
\end{equation}
The proof of $(\ref{2-form})$ is in \cite[Lemma 3.3]{Foscolo}. Moreover, for any vector fields $X$ and $Y$, the identity holds:
\begin{equation} \label{2-form action killing2}
X \cdot Y \cdot \kappa = - g(X,Y)\kappa + \omega(X,Y)\kappa + A_X Y \cdot \kappa.
\end{equation}

\subsection{The curvature endomorphism} \label{curvatureLaplacian}
In this subsection, we consider several operators related to the curvatures for the Levi-Civita connection $\nabla$ and the canonical Hermitian connection $\bar{\nabla}$. Furthermore, we introduce properties of these operators and relationships between them. Please refer to \cite{MoroianuSemmelmann} for detailed definitions and proofs. Note that in the remaining part of this paper, we adopt the Einstein convention of summation on the repeated subscripts. 

We review the definition of the basic curvature tensor. The curvature tensor is defined by
\[
R(W,X)Y \coloneqq \nabla_W \nabla_X Y - \nabla_X \nabla_W Y - \nabla_{[W,X]}Y \quad {\rm for \: all} \; W,X,Y \in \Gamma(TM).
\]
Also, for a vector field $Z$, we denote $g(R(W,X)Y,Z)$ by $R_{WXYZ}$. Replacing the Levi-Civita connection $\nabla$ with the canonical Hermitian connection $\bar{\nabla}$, we define the curvature tensor $\bar{R}$. The following relation holds for these two curvature tensors $R$ and $\bar{R}$. 

\begin{lemma}[\cite{Gray}]
For any vector fields $W,X,Y,Z$, one has 
\begin{equation} \label{hermitian curvature difference}
 \begin{split}
 R_{WXYZ} &= \bar{R}_{WXYZ} - \frac{1}{4} \Big( g(Y,W)g(X,Z) - g(X,Y)g(Z,W) \\
 &\quad - 3g(Y,JW)g(JX,Z) + 3g(Y,JX)g(JW,Z) + 2g(X,JW)g(JY,Z) \Big).
 \end{split}
\end{equation}
\end{lemma}

From the equation $(\ref{hermitian curvature difference})$ and the fact that $(M^6,g,J)$ is an Einstein manifold with $\Ric = 5g$, we see that the Ricci curvature of $\bar{R}$ satisfies $\overline{\Ric} = 4g$. Also, using the formula $(\ref{hermitian curvature difference})$ and the first Bianchi identity, for any vector fields $X,Y,Z$, we have 
\begin{equation} \label{hermitian Bianchi identity}
 \begin{split}
 &\bar{R}(X,Y)Z + \bar{R}(Y,Z)X + \bar{R}(Z,X)Y \\
 &= 2 \left( g(JX,Y)JZ + g(JY,Z)JX + g(JZ,X)JY \right).
 \end{split}
\end{equation}

Next, we write the curvature operator $R: \wedge^2 M \rightarrow \wedge^2 M$ properly:
\[
R(e_i \wedge e_j) = \frac{1}{2} R_{ijkl}e_k \wedge e_l = \frac{1}{2} e_k \wedge R(e_i,e_j)e_k,
\]
where $\{ e_i \}$ is a local orthonormal frame of $TM$. Let $EM$ be a vector bundle associated to the oriented orthonormal frame bundle or the spin bundle on $(M,g)$. We now define the curvature endomorphism $q(R) = q_E(R) \in \End (EM)$ as
\begin{equation} \label{curvature endo EM}
q(R) = \frac{1}{2} (e_i \wedge e_j)_{\star}R(e_i \wedge e_j)_{\star}.
\end{equation}
In particular, we have $q(R) = \Ric = 5\Id$ on 1-forms.

We can consider the curvature endomorphism for the canonical Hermitian connection $q(\bar{R})$. It is straightforward to show that the curvature endomorphism $q(\bar{R})$ satisfies $q(\bar{R}) = \overline{\Ric} = 4\Id$ on 1-forms. An interesting property of the curvature endomorphism $q(\bar{R})$ is that it preserves all tensor bundles associated to $\SU(3)$-representations. Refer to \cite[p.3061]{MoroianuSemmelmann} for the reason. 

\subsection{The actions of $A$} \label{action 2-tensors}
Let $X$ be a vector field, and $A_X$ be the tensor field $A_X = J(\nabla_X J)$. In Subsection \ref{nearly Kahler}, we introduced the extension $A_{X \star}$ to tensor fields $\Gamma(\End \mathcal{T}M)$. The result of Moroianu and Semmelmann \cite[Lemma 4.3]{MoroianuSemmelmann} gives the following equations related to the extension $A_{X \star}$ and the canonical Hermitian connection $\bar{\nabla}$ on $\wedge^{(1,1)}_0 M$, $\Sym^+_0 M$, $\Sym^- M$, and $\wedge^{(2,1)+(1,2)}_0 M$. We will essentially use these equations in the proof of Theorem \ref{thmA}.

\begin{proposition}[\cite{MoroianuSemmelmann}] \label{lem:herm}
Let $\varphi$ and $S$ be sections of $\wedge^{(1,1)}_0 M$ and $\Sym^- M$, respectively. Sections $h$
and $\sigma$ are defined in Subsection \ref{Algebraic results} by $g(Jh \cdot, \cdot) \coloneqq \varphi(\cdot, \cdot)$ and $\sigma \coloneqq S_{\star}\psi^+$. Then we have
 \begin{eqnarray} 
 A_{e_i \star}\bar{\nabla}_{e_i}\varphi &=& - (J\delta \varphi) \mathbin{\lrcorner} \psi^+ . \label{herm wedge^{(1,1)}_0} \\
 A_{e_i \star}\bar{\nabla}_{e_i}\sigma &=& - 2\delta S \wedge \omega . \label{herm wedge^{(2,0)+(0,2)}} \\
 (A_{e_i \star}\bar{\nabla}_{e_i}h)_{\star}\psi^+ &=& 2\delta h \wedge \omega - 4d\varphi . \label{herm Sym^+} \\
 A_{e_i \star}\bar{\nabla}_{e_i}S &=& (\delta S \mathbin{\lrcorner} \psi^+ + \delta\sigma) \circ J . \label{herm Sym^-}
 \end{eqnarray}
Here $\delta$ denotes the co-differential on differential forms or the divergence operator whenever applied to symmetric endomorphisms.
\end{proposition}

We remark that Moroianu and Semmelmann \cite{MoroianuSemmelmann} make a mistake regarding a formula $(\ref{herm Sym^+})$, but it does not affect the results in \cite{MoroianuSemmelmann} and ours.

Next, we introduce the other extension $\widetilde{A_{X \star}}$ to 2-tensor fields, which also plays an important role in the proof of Theorem \ref{thmA}. The extension $A_{X \star}$ is expressed for a 2-tensor field $\alpha \otimes \beta \in \Gamma(TM \otimes TM)$ as
\[
A_{X \star}(\alpha \otimes \beta) = A_X \alpha \otimes \beta + \alpha \otimes A_X \beta.
\]
On the other hand, we define $\widetilde{A_{X \star}}$ for $\alpha \otimes \beta \in \Gamma(TM \otimes TM)$ as
\[
\widetilde{A_{X \star}}(\alpha \otimes \beta) = A_X \alpha \otimes \beta - \alpha \otimes A_X \beta.
\]
To make formulas for $\widetilde{A_{X \star}}$ similar to those in Proposition \ref{lem:herm}, we prepare various tools.

First, for vector fields $X$, $Y$, $Z$, the identity holds:
\[
A_Z A_X Y = -g(X, Z)Y + g(Y, Z)X + g(JX, Z)JY - g(JY, Z)JX.
\]
This identity yields
\begin{equation} \label{AA}
A_{A_X Y} = X \wedge Y - JX \wedge JY.
\end{equation}

Next, we consider the induced action of a endomorphism $B \in \Gamma(\End TM)$ introduced in Subsection \ref{nearly Kahler}. The induced action of $B$ on a p-form $u$ is written as
\begin{equation} \label{induced action on form}
B_{\star}u = -B^{\ast}(e_i) \wedge e_i \mathbin{\lrcorner} u,
\end{equation}
where $B^{\ast}$ is the metric adjoint of $B$ and $\{ e_i \}$ is a local orthonormal frame of $TM$. Taking $A_X$ as $B$ and $\psi^+$ as $u$, we get by a easy calculation (cf. \cite[equation $(2.9)$]{MoroianuSemmelmann})
\begin{equation} \label{induced action on psi}
A_{X \star}\psi^+ = -2 X \wedge \omega.
\end{equation}
Taking $J$ as $B$, we have $(J_{\star})^2 = -(q-p)^2$ on $\Omega^{(p,q)+(q,p)}M$. By this characterization, we find that
\begin{eqnarray*}
\Gamma(\Sym^+ M) \ni h &\mapsto& h_{\star}\psi^{\pm} \in \Omega^{(3,0)+(0,3)}M, \\
\Gamma(\Sym^- M) \ni S &\mapsto& S_{\star}\psi^{\pm} \in \Omega^{(2,1)+(1,2)}M, \\
\Omega^{(1,1)}M \ni w &\mapsto& w_{\star}\psi^{\pm} \in \Omega^{(3,0)+(0,3)}M.
\end{eqnarray*}
A section $h_{\star}\psi^{\pm} \in \Omega^{(3,0)+(0,3)}M$ is expressed as $a\psi^+ + b\psi^-$ at each point by using some constants $a,b$. By substituting some bases and calculating $a,b$, we get
\begin{equation}
h_{\star}\psi^{\pm} = - \frac{1}{2} (\tr h) \psi^{\pm}.
\end{equation}
In particular, for a section $h$ of $\Sym^+_0 M$, we have
\begin{equation} \label{induced action h}
h_{\star}\psi^{\pm} = 0.
\end{equation}
In the same fashion, we have
\begin{equation} \label{induced action w}
w_{\star}\psi^{\pm} = 0.
\end{equation}

We consider the Hodge star operator $\ast$. The action of $\ast$ on $S_{\star}\psi^{\pm}$ is shown in \cite[equation (12)]{MoroianuNagySemmelmann}:
\begin{equation} \label{induced action S}
\ast (S_{\star}\psi^+) = - S_{\star}\psi^-.
\end{equation}
For a primitive $(1,1)$-form $\gamma$, we have
\begin{equation} \label{Hodge star wedge^{(1,1)}_0}
\ast(e_i \wedge \gamma) = - e_i \mathbin{\lrcorner} \ast \gamma = e_i \mathbin{\lrcorner} (\gamma \wedge \omega) = (e_i \mathbin{\lrcorner} \gamma) \wedge \omega + Je_i \wedge \gamma.
\end{equation}
Here, refer \cite[Lemma 2.3]{Foscolo} for the reason why $\ast \gamma = - \gamma \wedge \omega$ is valid for $\gamma \in \Omega^{(1,1)}_0 M$.

Finally, by applying Schur's Lemma to $\SU(3)$-decompositions, we get some relations (cf. \cite[Lemma 4.2]{MoroianuSemmelmann}).
\begin{lemma}
The following relations hold:
\begin{eqnarray}
e_i \wedge (A_{e_i \star}w) &=& 0 \quad {\rm for \: all} \; w \in \wedge^{(1,1)}_0 M, \label{Schur w} \\
e_i \mathbin{\lrcorner} (A_{e_i \star}w) &=& 0 \quad {\rm for \: all} \; w \in \wedge^{(1,1)}_0 M, \label{Schur w'} \\
(A_{e_i \star}h)(e_i) &=& 0 \quad {\rm for \: all} \; h \in \Sym^+_0 M, \label{Schur h} \\
(A_{e_i \star}S)(e_i) &=& 0 \quad {\rm for \: all} \; S \in \Sym^- M, \label{Schur S} \\
e_i \mathbin{\lrcorner} (A_{e_i \star}(S_{\star}\psi^-)) &=& 0 \quad {\rm for \: all} \; S \in \Sym^- M. \label{Schur S'}
\end{eqnarray}
\end{lemma}

Now, we have all the tools to prove the proposition below.
\begin{proposition} \label{prop; action 2-tensor}
Let $w$, $h$, and $S$ be sections of $\wedge^{(1,1)}_0 M$, $\Sym^+_0 M$, and $\Sym^- M$, respectively. A section $\sigma$ is defined by $\sigma = S_{\star}\psi^+$. Then we have
 \begin{eqnarray}
 (\widetilde{A_{e_i \star}}\bar{\nabla}_{e_i}w)_{\star}\psi^+ &=& 2\delta w \wedge \omega + 4 \ast dw. \label{second induced action wedge^{(1,1)}_0} \\
 \widetilde{A_{e_i \star}}\bar{\nabla}_{e_i}h &=& - \delta h \mathbin{\lrcorner} \psi^-. \label{second induced action Sym^+_0} \\
 \widetilde{A_{e_i \star}}\bar{\nabla}_{e_i}S &=& \ast d\sigma - \delta S \mathbin{\lrcorner} \psi^-. \label{second induced action Sym^-}
 \end{eqnarray}
\end{proposition}
\begin{proof}
Before checking $(\ref{second induced action wedge^{(1,1)}_0})$, we first calculate
\begin{eqnarray}
 \ast dw &=& \ast (e_i \wedge \nabla_{e_i}w) \overset{(\ref{Schur w})}{=} \ast (e_i \wedge \bar{\nabla}_{e_i}w) \nonumber \\
 &\overset{(\ref{Hodge star wedge^{(1,1)}_0}),(\ref{Schur w'})}{=}& Je_i \wedge \bar{\nabla}_{e_i}w - \delta w \wedge \omega. \label{ast d w}
\end{eqnarray}
Let us prove the equation $(\ref{second induced action wedge^{(1,1)}_0})$.
\begin{eqnarray*}
 (\widetilde{A_{e_i \star}}\bar{\nabla}_{e_i}w)_{\star}\psi^+ &=& - (A_{e_i} \circ \bar{\nabla}_{e_i}w)_{\star}\psi^+ - (\bar{\nabla}_{e_i}w \circ A_{e_i})_{\star}\psi^+ \\
 &=& - (A_{e_i \star}\bar{\nabla}_{e_i}w)_{\star}\psi^+ - 2(\bar{\nabla}_{e_i}w \circ A_{e_i})_{\star}\psi^+ \\
 &\overset{(\ref{herm wedge^{(1,1)}_0}),(\ref{induced action on form})}{=}& (J\delta w \mathbin{\lrcorner} \psi^+)_{\star}\psi^+ + 2A_{e_i}\bar{\nabla}_{e_i}w(e_k) \wedge e_k \mathbin{\lrcorner} \psi^+\\
 &=& -A_{\delta w \star}\psi^+ + 2A_{e_i \star}(\bar{\nabla}_{e_i}w(e_k) \wedge e_k \mathbin{\lrcorner}\psi^+) \\
 &&-2\bar{\nabla}_{e_i}w(e_k) \wedge A_{e_i \star}(e_k \mathbin{\lrcorner} \psi^+) \\
 &\overset{(\ref{induced action on form}),(\ref{induced action on psi})}{=}& 2\delta w \wedge \omega + 2A_{e_i \star}((\bar{\nabla}_{e_i}w)_{\star}\psi^+) \\
 && -2\bar{\nabla}_{e_i}w(e_k) \wedge (-A_{e_i}e_k \mathbin{\lrcorner} \psi^+ + e_k \mathbin{\lrcorner} A_{e_i \star}\psi^+) \\
 &\overset{(\ref{induced action on psi})}{=}& 2\delta w \wedge \omega -2\bar{\nabla}_{e_i}w(e_k) \wedge JA_{A_{e_i}e_k} + 4(\bar{\nabla}_{e_i}w)_{\star}(e_i \wedge \omega) \\
 &\overset{(\ref{AA})}{=}& 2\delta w \wedge \omega -2\bar{\nabla}_{e_i}w(e_k) \wedge (e_i \wedge Je_k + Je_i \wedge e_k) \\
 && + 4(\bar{\nabla}_{e_i}w)e_i \wedge \omega + 4e_i \wedge (\bar{\nabla}_{e_i}w)_{\star}\omega \\
 &\overset{(\ref{induced action w}),(\ref{Schur w'})}{=}& 4Je_i \wedge \bar{\nabla}_{e_i}w - 2\delta w \wedge \omega \\
 &\overset{(\ref{ast d w})}{=}& 2\delta w \wedge \omega + 4 \ast dw.
\end{eqnarray*}

Next, for any vector field $X$, we get
\begin{eqnarray*}
 (\widetilde{A_{e_i \star}}\bar{\nabla}_{e_i}h)(X) &=& - A_{e_i}\bar{\nabla}_{e_i}h(X) - \bar{\nabla}_{e_i}h(A_{e_i}X) \\
 &=& -(\bar{\nabla}_{e_i}h)(X) \mathbin{\lrcorner} e_i \mathbin{\lrcorner} \psi^- - (\bar{\nabla}_{e_i}h)(X \mathbin{\lrcorner} e_i \mathbin{\lrcorner} \psi^-) \\
 &=& (\bar{\nabla}_{e_i}h)(e_j,X) e_i \mathbin{\lrcorner} e_j \mathbin{\lrcorner} \psi^- - \psi^-(e_i,X,e_j)(\bar{\nabla}_{e_i}h)e_j \\
 &=& ((\bar{\nabla}_{e_i}h)e_j \wedge (e_i \mathbin{\lrcorner} e_j \mathbin{\lrcorner} \psi^-))(X).
\end{eqnarray*}
Thus, we have
\[
\widetilde{A_{e_i \star}}\bar{\nabla}_{e_i}h = (\bar{\nabla}_{e_i}h)e_j \wedge (e_i \mathbin{\lrcorner} e_j \mathbin{\lrcorner} \psi^-).
\]
We proceed further with the calculation.
\begin{eqnarray*}
 \widetilde{A_{e_i \star}}\bar{\nabla}_{e_i}h &=& -e_i \mathbin{\lrcorner} ((\bar{\nabla}_{e_i}h)e_j \wedge e_j \mathbin{\lrcorner} \psi^-) + (e_i \mathbin{\lrcorner} (\bar{\nabla}_{e_i}h)e_j)e_j \mathbin{\lrcorner} \psi^- \\
 &\overset{(\ref{Schur h})}{=}& - e_i \mathbin{\lrcorner} \bar{\nabla}_{e_i} (h(e_j) \wedge e_j \mathbin{\lrcorner} \psi^-) - \delta h (e_j) e_j \mathbin{\lrcorner} \psi^- \\
 &=& e_i \mathbin{\lrcorner} \bar{\nabla}_{e_i} (h_{\star}\psi^-) - \delta h \mathbin{\lrcorner} \psi^- \\
 &\overset{(\ref{induced action h})}{=}& -\delta h \mathbin{\lrcorner} \psi^-.
\end{eqnarray*}
This proves $(\ref{second induced action Sym^+_0})$.

In the same way that we prove the identity $(\ref{second induced action Sym^+_0})$, we know that
\begin{equation} \label{calc second induced action Sym^-}
\widetilde{A_{e_i \star}}\bar{\nabla}_{e_i}S = (\bar{\nabla}_{e_i}S)e_j \wedge (e_i \mathbin{\lrcorner} e_j \mathbin{\lrcorner} \psi^-) \overset{(\ref{Schur S}),(\ref{Schur S'})}{=} -\delta(S_{\star}\psi^-) - \delta S \mathbin{\lrcorner} \psi^-.
\end{equation}
The first term of the equation $(\ref{calc second induced action Sym^-})$ is
\begin{eqnarray*}
 \delta(S_{\star}\psi^-) = - \ast d \ast (S_{\star}\psi^-) \overset{(\ref{induced action S})}{=} \ast d \ast^2 (S_{\star}\psi^+) = - \ast d \sigma.
\end{eqnarray*}
Thus, we obtain the equation $(\ref{second induced action Sym^-})$.
\end{proof}

\section{Comparison of differential operators} \label{differential operators}
Assume that $(M^6,g,J)$ is a 6-dimensional strict nearly K\"{a}hler manifold with the normalized scalar curvature $\scal = 30$. We defined the twisted Dirac operator $D_{TM}$ in Section \ref{Preliminalies}. Then we have a natural 2nd order differential operator: the standard Laplace operator $\Delta \coloneqq \nabla^{\ast}\nabla + q(R)$, that is the sum of the rough Laplacian and the curvature endomorphism in $(\ref{curvature endo EM})$. Similarly, we define the twisted Dirac operator for the canonical Hermitian connection as $\overline{D_{TM}} \coloneqq (e_k \cdot \otimes \Id_{TM^{\C}}) \circ \bar{\nabla}_{e_k}$ and the Hermitian Laplace operator as $\bar{\Delta} \coloneqq \bar{\nabla}^{\ast}\bar{\nabla} + q(\bar{R})$. These Laplace operators are introduced in \cite{MoroianuSemmelmann2} and \cite{SemmelmannWeingart}. We shall study the relationship between the twisted Dirac operator and the Hermitian Laplace operator.

We get the difference below between the twisted Dirac operator for the Levi-Civita connection and the one for the canonical Hermitian connection.

\begin{theorem}
On the sections of $S_{1/2} \otimes TM$, the relation holds:
\begin{equation} \label{twisted first}
 \overline{D_{TM}} = D_{TM} - \frac{3}{4}\psi^- \cdot \otimes \Id - \frac{1}{2} e_i \cdot \otimes A_{e_i}.
\end{equation}
\end{theorem}

\begin{proof}
The theorem is a straightforward consequence of $(\ref{herm conn TM})$ and $(\ref{herm conn S})$.
\end{proof}

We next see the relationship between the Hermitian Laplace operator and the twisted Dirac operator for the canonical Hermitian connection.

\begin{lemma}
The identity holds:
\begin{equation} \label{lichnerowicz lemma}
e_j \cdot \bar{R}_{S}(X,e_j)\zeta = -\frac{1}{2}\overline{\Ric}(X)\cdot \zeta - X\cdot \zeta + JX \cdot \omega \cdot \zeta,
\end{equation}
for any vector field $X$ and spinor $\zeta$, where $\bar{R}_{S}$ denotes the curvature tensor for the canonical Hermitian connection on the spinor bundle $S_{1/2}$.
\end{lemma}
\begin{proof}
Using the Clifford relation, we get
\begin{equation*}
 \begin{split}
 12 e_j \cdot \bar{R}_{S}(X,e_j) &= 3g(\bar{R}(X, e_j)e_k,e_l)e_je_ke_l = 3g(X,e_i)\bar{R}_{ijkl}e_je_ke_l \\
 &= g(X,e_i)(\bar{R}_{ijkl}e_je_ke_l + \bar{R}_{iklj}e_ke_le_j + \bar{R}_{iljk}e_le_je_k) \\
 &= g(X,e_i)\bar{R}_{ijkl}e_je_ke_l + g(X,e_i)\bar{R}_{iklj}(-2\delta_{lj}e_k + 2\delta_{kj}e_l + e_je_ke_l) \\
 &\quad + g(X,e_i)\bar{R}_{iljk}(-2\delta_{lj}e_k + 2\delta_{lk}e_j + e_je_ke_l) \\
 &= g(X,e_i)(\bar{R}_{ijkl} + \bar{R}_{iklj} + \bar{R}_{iljk})e_je_ke_l - 6\overline{\Ric}(X,e_k)e_k.
 \end{split}
\end{equation*}
From the equation $(\ref{hermitian Bianchi identity})$, it is straightforward to show that
\begin{equation*}
 \begin{split}
 &g(X,e_i)(\bar{R}_{ijkl} + \bar{R}_{iklj} + \bar{R}_{iljk})e_je_ke_l \\
 &= -2g(Je_j,e_k)g(Je_l,X)e_je_ke_l - 2g(Je_k,e_l)g(Je_j,X)e_je_ke_l \\
 &\quad - 2g(Je_l,e_j)g(Je_k,X)e_je_ke_l \\ 
 &= -12X + 12JX \cdot \omega.
 \end{split}
\end{equation*}
Combining the above equations, we arrive at the required identity $(\ref{lichnerowicz lemma})$.
\end{proof}

The next equation follows immediately from $(\ref{lichnerowicz lemma})$.
\begin{equation} \label{lichnerowicz lemma2}
e_i \cdot e_j \cdot \bar{R}_{S}(e_i,e_j) = 18+2\omega \cdot \omega \cdot .
\end{equation}
From this equation, for the curvature endomorphism $q_S(\bar{R})$ on the spinor bundle $S_{1/2}$, it follows that
\[
q_S(\bar{R}) = \frac{9}{2} + \frac{1}{2}\omega \cdot \omega \cdot.
\]

\begin{theorem}
Between the square of the twisted Dirac operator and the Hermitian Laplace operator, we obtain the relation:
\begin{equation} \label{Lichnerowicz}
\overline{D_{TM}}^2 = \bar{\Delta}_{S \otimes T} + \frac{1}{2} + \frac{1}{2}\omega \cdot \omega \cdot \otimes \Id + (e_j \mathbin{\lrcorner}\psi^- \cdot \otimes \Id)\bar{\nabla}_{e_j} .
\end{equation}
\end{theorem}

\begin{proof}
Calculating the same as the Lichnerowicz formula (cf. \cite[p.107]{Yasushi}), we get
\[
\overline{D_{TM}}^2 = \bar{\nabla}^{\ast}\bar{\nabla} + (e_j \mathbin{\lrcorner} \psi^- \cdot \otimes \Id)\bar{\nabla}_{e_j} + \frac{1}{2} e_je_k \{ \bar{R}_{S}(e_j,e_k) \otimes \Id + \Id \otimes \bar{R}(e_j,e_k) \}.
\]
Applying the equation $(\ref{lichnerowicz lemma2})$, we find that the above equation becomes
\begin{equation} \label{Lichnerowicz''}
\overline{D_{TM}}^2 = \bar{\nabla}^{\ast}\bar{\nabla} +9 + \omega \cdot \omega \cdot \otimes \Id + (e_j \mathbin{\lrcorner} \psi^- \otimes \Id)\bar{\nabla}_{e_j} + \frac{1}{2} e_je_k \otimes \bar{R}(e_j,e_k).
\end{equation}
Next, we compute the curvature endomorphism $q_{S \otimes T}(\bar{R})$, which is the curvature term of the Hermitian Laplace operator $\bar{\Delta}_{S \otimes T}$.
\begin{equation*}
 \begin{split}
 q_{S \otimes T}(\bar{R}) &= \frac{1}{2}\bar{R}(e_i \wedge e_j)_{\star} \otimes (e_i \wedge e_j)_{\star} + \frac{1}{2} (e_i \wedge e_j)_{\star} \otimes \bar{R}(e_i \wedge e_j)_{\star} \\
 &\quad + q_S(\bar{R}) \otimes \Id + \Id \otimes q_T(\bar{R}) \\
 &= \frac{9}{2} + \frac{1}{2}\omega \cdot \omega \otimes \Id + \frac{1}{2}e_je_k \otimes \bar{R}(e_j,e_k) + \Id \otimes \overline{\Ric}.
 \end{split}
\end{equation*}
Here, we use the equations $q_S(\bar{R}) = \frac{9}{2} + \frac{1}{2} \omega \cdot \omega \cdot$ and $q_T(\bar{R}) = \overline{\Ric} = 4 \Id$. Substituting the above equation into the equation $(\ref{Lichnerowicz''})$, we obtain the relation $(\ref{Lichnerowicz})$.
\end{proof}

The following lemma is obtained directly from $(\ref{herm conn TM})$, $(\ref{herm conn S})$, and the definition of the rough Laplacian $\nabla^{\ast}\nabla = - \nabla_{e_i}\nabla_{e_i} + \nabla_{\nabla_{e_i}e_i}$.

\begin{lemma} \label{lem:rough laplacian2}
On the sections of $S_{1/2} \otimes TM$, we have
\begin{equation} \label{rough laplacian2}
\bar{\nabla}^{\ast}\bar{\nabla} = \nabla^{\ast}\nabla + \frac{5}{8} + \frac{1}{8}\omega\cdot\omega\cdot \otimes \Id + (\Id \otimes A_{e_i})\nabla_{e_i} + \frac{1}{2}(e_i \mathbin{\lrcorner} \psi^- \cdot\otimes\Id)\bar{\nabla}_{e_i}.
\end{equation}
\end{lemma}

Finally, we see the difference in the square of the twisted Dirac operator for each connection.

\begin{theorem}
On the sections of $S_{1/2} \otimes TM$, we obtain the relation between the squares of the two twisted Dirac operators $D_{TM}$ and $\overline{D_{TM}}$:
\begin{equation} \label{twisted second}
 \begin{split}
 \overline{D_{TM}}^2 =& {D_{TM}}^2 + \frac{17}{8} + \frac{9}{8}\omega \cdot \omega \cdot \otimes \Id + \frac{3}{2} (e_j \mathbin{\lrcorner} \psi^{-} \cdot \otimes \Id)\bar{\nabla}_{e_j} \\
 &+ (\Id \otimes A_{e_j})\nabla_{e_j} + \frac{1}{2} e_j \cdot e_k \cdot \otimes (\bar{R}(e_j,e_k) - R(e_j,e_k)).
 \end{split}
\end{equation}
In particular, for a local section $\alpha^{(i)} \otimes e_i$ of $S_{1/2} \otimes TM$, we have
\begin{equation*}
 \begin{split}
 \overline{D_{TM}}^2(\alpha^{(i)} \otimes e_i) &= {D_{TM}}^2(\alpha^{(i)} \otimes e_i) - \frac{7}{8}\alpha^{(i)} \otimes e_i + \frac{9}{8}\omega\cdot\omega\cdot\alpha^{(i)} \otimes e_i + \frac{1}{2}\omega\cdot\alpha^{(i)} \otimes Je_i \\
 &\quad +\frac{3}{2}(e_j \mathbin{\lrcorner} \psi^- \cdot \otimes \Id)\bar{\nabla}_{e_j}(\alpha^{(i)} \otimes e_i) + (\Id\otimes A_{e_j})\bar{\nabla}_{e_j}(\alpha^{(i)} \otimes e_i) \\
 &\quad + \frac{1}{4}(e_j \mathbin{\lrcorner} \psi^- \cdot \otimes A_{e_j})(\alpha^{(i)} \otimes e_i) \\
 &\quad - \frac{1}{4} e_j \cdot e_i \cdot \alpha^{(i)} \otimes e_j - \frac{3}{4} Je_i \cdot Je_j \cdot \alpha^{(i)} \otimes e_j.
 \end{split}
\end{equation*}
\end{theorem}

\begin{proof}
Due to \cite[p.107]{Yasushi}, the Lichnerowicz formula for the twisted Dirac operator is 
\[
{D_{TM}}^2 = \nabla^{\ast}\nabla + \frac{15}{2} + \frac{1}{2}e_je_k \otimes R(e_j,e_k).
\]
Employing the equations $(\ref{Lichnerowicz})$ and $(\ref{rough laplacian2})$, we obtain
\begin{equation*}
 \begin{split}
 \overline{D_{TM}}^2 &= \bar{\nabla}^{\ast}\bar{\nabla} + 9 + \omega\cdot\omega\cdot \otimes \Id + (e_j \mathbin{\lrcorner} \psi^- \cdot \otimes \Id)\bar{\nabla}_{e_j} + \frac{1}{2} e_je_k \otimes \bar{R}(e_j,e_k) \\
 &= \nabla^{\ast}\nabla + \frac{77}{8} + \frac{9}{8}\omega\cdot\omega\cdot\otimes\Id + \frac{3}{2}(e_j \mathbin{\lrcorner} \psi^- \otimes \Id)\bar{\nabla}_{e_j} \\
 &\quad + \frac{1}{2} e_je_k \otimes \bar{R}(e_j,e_k) + (\Id\otimes A_{e_j})\nabla_{e_j} \\
 &= {D_{TM}}^2 + \frac{17}{8} + \frac{9}{8}\omega\cdot\omega\otimes\Id + \frac{3}{2}(e_j \mathbin{\lrcorner} \psi^- \otimes \Id)\bar{\nabla}_{e_j} \\
 &\quad + (\Id\otimes A_{e_j})\nabla_{e_j} + \frac{1}{2} e_je_k \otimes (\bar{R}(e_j,e_k)-R(e_j,e_k)).
 \end{split}
\end{equation*}
We calculate the curvature term in the above equation by using $(\ref{hermitian curvature difference})$. For $\alpha^{(i)} \otimes e_i$ in $\Gamma(S_{1/2}\otimes TM)$, we have
\begin{equation*}
 \begin{split}
 &\frac{1}{2} e_je_k \alpha^{(i)} \otimes(\bar{R}(e_j,e_k)e_i-R(e_j,e_k)e_i) \\
 &= \frac{1}{8}e_je_k \alpha^{(i)} \otimes \big{\{} \delta_{ji}e_k - \delta_{ki}e_j - 3g(Je_j,e_i)Je_k + 3g(Je_k,e_i)Je_j + 2g(Je_j,e_k)e_i \big{\}} \\
 &= \frac{1}{8} e_i \cdot e_j \cdot \alpha^{(i)} \otimes e_j - \frac{1}{8} e_j \cdot e_i \cdot \alpha^{(i)} \otimes e_j + \frac{3}{8} Je_i \cdot e_j \cdot \alpha^{(i)} \otimes Je_j \\
 &\quad - \frac{3}{8} e_j \cdot Je_i \cdot \alpha^{(i)} \otimes Je_j + \frac{1}{4} e_j \cdot Je_j \cdot \alpha^{(i)} \otimes Je_i \\
 &= -\frac{1}{4} e_j \cdot e_i \cdot \alpha^{(i)} \otimes e_j - \frac{3}{4} Je_i \cdot Je_j \cdot \alpha^{(i)} \otimes e_j - \alpha^{(i)} \otimes e_i + \frac{1}{2}\omega \cdot \alpha^{(i)} \otimes Je_i.
 \end{split}
\end{equation*}
This gives the theorem.
\end{proof}

\section{Rarita-Schwinger fields on nearly K\"{a}hler manifolds} \label{Rarita Schwinger fields on nearly Kahler}
The goal of this section is to find the conditions under which Rarita-Schwinger fields exist, and this is the main result in this paper.

We recall that the definition of a section $\phi \in \Gamma(S_{1/2} \otimes TM)$ being the Rarita-Schwinger field is $\phi \in \Gamma(S_{3/2})$ and $D_{TM}\phi=0$. These are, of course, equivalent to $\phi \in \Gamma(S_{3/2})$ and ${D_{TM}}^2\phi=0$.

Any elements of $\Gamma(S_{1/2} \otimes TM)$ are represented locally as $\alpha^{(i)} \otimes e_i$ using a local orthonormal frame $\{ e_i \}$. By the isomorphism $(\ref{irrep decomp S_{frac{1}{2}}})$, let $\alpha^{(i)} \in \Gamma(S_{1/2})$ be decomposed into 
\begin{equation} \label{alpha decomposition}
\alpha^{(i)} = \left( {\alpha_0}^{(i)} + {\alpha_1}^{(i)} + {\alpha_6}^{(i)}\vol \right) \cdot \kappa \in \left( \Omega^0M \oplus \Omega^1M \oplus \Omega^6M \right) \cdot \kappa.
\end{equation}
We use this decomposition liberally in the calculations of this section.

First, we rewrite the condition that $\alpha^{(i)} \otimes e_i \in \Gamma(S_{1/2} \otimes TM)$ is an element of $\Gamma(S_{3/2})$.

\begin{lemma} \label{S_{3/2} condition}
Let $\alpha^{(i)} \otimes e_i$ be in $S_{1/2} \otimes TM$. Then, $\alpha^{(i)} \otimes e_i$ is in $S_{3/2}$ if and only if 
\begin{numcases}
{}
\label{S_{3/2} condition omega^0} {\alpha_1}^{(i)} \odot e_i \in \Sym_0 M, & \\ 
\label{S_{3/2} condition omega^6} {\alpha_1}^{(i)} \wedge e_i \in \wedge^{(1,1)}_0 M \oplus \wedge^{(2,0)+(0,2)}M, & \\
\label{S_{3/2} condition omega^1} {\alpha_0}^{(i)} e_i - {\alpha_6}^{(i)}Je_i + A_{e_i}{\alpha_1}^{(i)}=0,
\end{numcases}
where $\Sym_0M$ denotes the trace-free part of $\Sym M$.
\end{lemma}

\begin{proof}
By definition, for $\alpha^{(i)} \otimes e_i$ in $S_{3/2}$, we know $e_i \cdot \alpha^{(i)}=0$. Using the equations $(\ref{5-form})$ and $(\ref{2-form action killing2})$, we get
\begin{equation*}
 \begin{split}
 e_i \cdot \alpha^{(i)} = -g(e_i,{\alpha_1}^{(i)})\kappa + \omega(e_i,{\alpha_1}^{(i)})\vol \cdot \kappa + \left( {\alpha_0}^{(i)}e_i - {\alpha_6}^{(i)}Je_i + A_{e_i}{\alpha_1}^{(i)} \right) \cdot \kappa.
 \end{split}
\end{equation*}
Using the isomorphism $(\ref{irrep decomp S_{frac{1}{2}}})$ again, we get $g(e_i,{\alpha_1}^{(i)})=0$, $\omega(e_i,{\alpha_1}^{(i)})=0$, and $ {\alpha_0}^{(i)}e_i - {\alpha_6}^{(i)}Je_i + A_{e_i}{\alpha_1}^{(i)}=0$. Here, the equations $g(e_i,{\alpha_1}^{(i)})=0$ and $\omega(e_i,{\alpha_1}^{(i)})=0$ are equivalent to $(\ref{S_{3/2} condition omega^0})$ and $(\ref{S_{3/2} condition omega^6})$, respectively. The converse is obvious from the above discussion.
\end{proof}

We already know the decomposition of the spinor bundle $S_{1/2}$ $(\ref{irrep decomp S_{frac{1}{2}}})$. Also, we have an $\SU(3)$ irreducible decomposition of $S_{1/2} \otimes TM$:
\begin{equation} \label{irrer decomp S_{1/2} otimes TM}
S_{1/2} \otimes TM \cong 2\wedge^1 M \oplus \R\omega \oplus \wedge^{(2,0)+(0,2)}M \oplus \wedge^{(1,1)}_0 M \oplus (\Sym^+_0 M \oplus \R g) \oplus \Sym^- M,
\end{equation}
where $\R g$ is the trace part of $\Sym^+ M$. Keep in mind that the two $\wedge^1 M$ above are derived from $\wedge^0M \otimes \wedge^1M$ and $\wedge^6M \otimes \wedge^1M$, respectively.

Next, we will see what conditions we get when $\alpha^{(i)} \otimes e_i$ is both in $\Gamma(S_{3/2})$ and in the kernel of the twisted Dirac operator. 

\begin{lemma} \label{lemma;alpha_0,alpha_6}
We denote a Rarita-Schwinger field $\phi$ locally as $\alpha^{(i)} \otimes e_i$, then we obtain ${\alpha_0}^{(i)}e_i=0$ and ${\alpha_6}^{(i)}e_i=0$.
\end{lemma}

\begin{proof}
We take a local orthonormal frame $\{ e_i \}$ which is $\bar{\nabla}$-parallel at a point. By the formula $(\ref{twisted first})$, for a Rarita-Schwinger field $\alpha^{(i)} \otimes e_i$, the following holds:
\begin{equation} \label{condition Rarita1}
\overline{D_{TM}}\left( \alpha^{(i)} \otimes e_i \right) + \frac{3}{4} \psi^- \cdot \alpha^{(i)} \otimes e_i + \frac{1}{2} e_j \cdot \alpha^{(i)} \otimes A_{e_j}e_i = 0.
\end{equation}
Using the decomposition $(\ref{alpha decomposition})$ and the formulas $(\ref{5-form})$-$(\ref{2-form action killing2})$, we compute each term in the left-hand side of the equation $(\ref{condition Rarita1})$. For the first term of the equation $(\ref{condition Rarita1})$, we have
\begin{eqnarray*}
\overline{D_{TM}}\left( {\alpha_0}^{(i)} \kappa \otimes e_i \right) &=& (e_j \cdot \otimes \Id)\bar{\nabla}_{e_j}(\kappa \otimes {\alpha_0}^{(i)}e_i), \\
\overline{D_{TM}}\left( {\alpha_6}^{(i)}\vol \cdot \kappa \otimes e_i \right) &=& - (Je_j \cdot \otimes \Id)\bar{\nabla}_{e_j}(\kappa \otimes {\alpha_6}^{(i)}e_i), \\
\overline{D_{TM}}\left( {\alpha_1}^{(i)} \cdot \kappa \otimes e_i \right) &=& -\bar{\nabla}_{e_j}(g(e_j, {\alpha_1}^{(i)})\kappa \otimes e_i) - \bar{\nabla}_{e_j}(g(e_j, J{\alpha_1}^{(i)})\vol \cdot \kappa \otimes e_i) \\
 && + \bar{\nabla}_{e_j}(A_{e_j}{\alpha_1}^{(i)} \cdot \kappa \otimes e_i) + g(\bar{\nabla}_{e_j}e_j, {\alpha_1}^{(i)})\kappa \otimes e_i \\
 && + g(\bar{\nabla}_{e_j}e_j, J{\alpha_1}^{(i)})\vol \cdot \kappa \otimes e_i - A_{\bar{\nabla}_{e_j}e_j}{\alpha_1}^{(i)} \cdot \kappa \otimes e_i.
\end{eqnarray*}
Similarly, for the second term of the equation $(\ref{condition Rarita1})$, we have
\begin{eqnarray*}
\psi^- \cdot {\alpha_0}^{(i)} \kappa \otimes e_i &=& -4 \kappa \otimes {\alpha_0}^{(i)}e_i, \\
\psi^- \cdot {\alpha_6}^{(i)}\vol \cdot \kappa \otimes e_i &=& 4 \vol \cdot \kappa \otimes {\alpha_6}^{(i)}e_i, \\
\psi^- \cdot {\alpha_1}^{(i)} \cdot \kappa \otimes e_i &=& 0.
\end{eqnarray*}
Finally, for the third term of the equation $(\ref{condition Rarita1})$, we have
\begin{eqnarray*}
e_j \cdot {\alpha_0}^{(i)} \kappa \otimes A_{e_j}e_i &=& e_j \cdot \kappa \otimes {\alpha_0}^{(i)}A_{e_j}e_i, \\
e_j \cdot {\alpha_6}^{(i)}\vol \cdot \kappa \otimes A_{e_j}e_i &=& -Je_j \cdot \kappa \otimes {\alpha_6}^{(i)}A_{e_j}e_i, \\
e_j \cdot {\alpha_1}^{(i)} \cdot \kappa \otimes A_{e_j}e_i &=& \kappa \otimes A_{e_i}{\alpha_1}^{(i)} - \vol \cdot \kappa \otimes JA_{e_i}{\alpha_1}^{(i)} \\
&& - e_i \cdot \kappa \otimes {\alpha_1}^{(i)} - Je_i \cdot \kappa \otimes J{\alpha_1}^{(i)}.
\end{eqnarray*}
Summarizing the above calculations, we find that the equation $(\ref{condition Rarita1})$ becomes
\begin{equation} \label{D_{TM}=0}
 \begin{split}
 0&= \kappa \otimes \left( -3{\alpha_0}^{(i)}e_i - \bar{\nabla}_{e_j}(g(e_j, {\alpha_1}^{(i)})e_i) + \frac{1}{2}A_{e_i}{\alpha_1}^{(i)} \right) \\
 &\quad + \vol \cdot \kappa \otimes \left( 3{\alpha_6}^{(i)}e_i - \bar{\nabla}_{e_j}(g(e_j, J{\alpha_1}^{(i)})e_i) - \frac{1}{2}JA_{e_i}{\alpha_1}^{(i)} \right) \\
 &\quad + \bar{\nabla}_{e_j}(e_j \cdot \kappa \otimes {\alpha_0}^{(i)}e_i) + \frac{1}{2} e_j \cdot \kappa \otimes {\alpha_0}^{(i)}A_{e_j}e_i - \bar{\nabla}_{e_j}(Je_j \cdot \kappa \otimes {\alpha_6}^{(i)}e_i) \\
 &\quad - \frac{1}{2} Je_j \cdot \kappa \otimes {\alpha_6}^{(i)}A_{e_j}e_i + (A_{e_j} \otimes \Id)\bar{\nabla}_{e_j}({\alpha_1}^{(i)} \cdot \kappa \otimes e_i) \\
 &\quad - \frac{1}{2} e_i \cdot \kappa \otimes {\alpha_1}^{(i)} - \frac{1}{2}Je_i \cdot \kappa \otimes J{\alpha_1}^{(i)}.
 \end{split}
\end{equation}
Our main idea is to project the equation $(\ref{D_{TM}=0})$ onto each space of the irreducible decomposition $(\ref{irrer decomp S_{1/2} otimes TM})$. First, projecting the equation $(\ref{D_{TM}=0})$ onto the bundle $\wedge^0M \otimes \wedge^1M \cong \wedge^1 M$, we get
\begin{eqnarray}
 0&=&-3{\alpha_0}^{(i)}e_i - \bar{\nabla}_{e_j}(g(e_j, {\alpha_1}^{(i)})e_i) + \frac{1}{2}A_{e_i}{\alpha_1}^{(i)} \nonumber \\
 &\overset{(\ref{S_{3/2} condition omega^1})}{=}& - \frac{7}{2}{\alpha_0}^{(i)}e_i + \frac{1}{2}{\alpha_6}^{(i)}Je_i - \bar{\nabla}_{e_j}(g(e_j,{\alpha_1}^{(i)})e_i). \label{D_{TM}=0, wedge^0 otimes wedge^1}
\end{eqnarray}
In the same way, projecting onto $\wedge^6M \otimes \wedge^1M \cong \wedge^1M$, we get
\begin{equation} \label{D_{TM}=0, wedge^6 otimes wedge^1}
0\overset{(\ref{S_{3/2} condition omega^1})}{=} \frac{7}{2}{\alpha_6}^{(i)}e_i + \frac{1}{2}{\alpha_0}^{(i)}Je_i - \bar{\nabla}_{e_j}(g(e_j,J{\alpha_1}^{(i)})e_i).
\end{equation}
As a reminder, using the symbols that will be defined later in $(\ref{w})$ and $(\ref{H})$, the term $\bar{\nabla}_{e_j}(g(e_j, {\alpha_1}^{(i)})e_i)$ in the equation $(\ref{D_{TM}=0, wedge^0 otimes wedge^1})$ is expressed as 
\[
\bar{\nabla}_{e_j}(g(e_j, {\alpha_1}^{(i)})e_i) = - \frac{1}{2}\delta({\alpha_1}^{(i)} \wedge e_i) - \frac{1}{2}\delta({\alpha_1}^{(i)} \odot e_i).
\]
Analogously, the term $\bar{\nabla}_{e_j}(g(e_j,J{\alpha_1}^{(i)})e_i)$ in the equation $(\ref{D_{TM}=0, wedge^6 otimes wedge^1})$ is expressed as 
\[
\bar{\nabla}_{e_j}(g(e_j,J{\alpha_1}^{(i)})e_i) = \frac{1}{2}J\delta(J{\alpha_1}^{(i)} \wedge Je_i) + \frac{1}{2}J\delta(J{\alpha_1}^{(i)} \odot Je_i).
\]
We leave the other projection in $(\ref{D_{TM}=0})$ to the proof of Theorem \ref{thmA}. 

Now let's consider another equation ${D_{TM}}^2 \left( \alpha^{(i)} \otimes e_i \right)=0$. According to $(\ref{Lichnerowicz})$ and $(\ref{twisted second})$, the equation ${D_{TM}}^2 \left( \alpha^{(i)} \otimes e_i \right)=0$ is equivalent to
\begin{equation} \label{D_{TM}^2=0 iff}
 \begin{split}
 \bar{\Delta} \left( \alpha^{(i)} \otimes e_i \right) &= - \frac{11}{8} \alpha^{(i)} \otimes e_i + \frac{5}{8} \omega \cdot \omega \cdot \alpha^{(i)} \otimes e_i + \frac{1}{2} \omega \cdot \alpha^{(i)} \otimes Je_i \\
 &\quad + \frac{1}{4} e_j \mathbin{\lrcorner} \psi^- \cdot \alpha^{(i)} \otimes A_{e_j}e_i - \frac{3}{4} Je_i \cdot Je_j \cdot \alpha^{(i)} \otimes e_j \\
 &\quad + \frac{1}{2} (e_j \mathbin{\lrcorner} \psi^- \cdot \otimes \Id)\bar{\nabla}_{e_j}(\alpha^{(i)} \otimes e_i) + (\Id \otimes A_{e_j})\bar{\nabla}_{e_j}(\alpha^{(i)} \otimes e_i).
 \end{split}
\end{equation}
Using the decomposition $(\ref{alpha decomposition})$ and the formulas $(\ref{5-form})$-$(\ref{2-form action killing2})$ again, we compute each term in the right-hand side of the equation $(\ref{D_{TM}^2=0 iff})$.
\begin{eqnarray*}
- \frac{11}{8} \alpha^{(i)} \otimes e_i &=& - \frac{11}{8} \kappa \otimes {\alpha_0}^{(i)}e_i - \frac{11}{8} \vol \cdot \kappa \otimes {\alpha_6}^{(i)}e_i \\
&&- \frac{11}{8} {\alpha_1}^{(i)} \cdot \kappa \otimes e_i, \\
\frac{5}{8} \omega \cdot \omega \cdot \alpha^{(i)} \otimes e_i &=& - \frac{45}{8} \kappa \otimes {\alpha_0}^{(i)}e_i - \frac{45}{8} \vol \cdot \kappa \otimes {\alpha_6}^{(i)}e_i \\
&&- \frac{5}{8} {\alpha_1}^{(i)} \cdot \kappa \otimes e_i, \\
\frac{1}{2} \omega \cdot \alpha^{(i)} \otimes Je_i &=& - \frac{3}{2} \kappa \otimes {\alpha_6}^{(i)}Je_i + \frac{3}{2} \vol \cdot \kappa \otimes {\alpha_0}^{(i)}Je_i\\
&& - \frac{1}{2} J{\alpha_1}^{(i)} \cdot \kappa \otimes Je_i, \\
\frac{1}{4} e_j \mathbin{\lrcorner} \psi^- \cdot \alpha^{(i)} \otimes A_{e_j}e_i &=& \frac{1}{2} \kappa \otimes A_{e_i}{\alpha_1}^{(i)} + \frac{1}{2} \vol \cdot \kappa \otimes JA_{e_i}{\alpha_1}^{(i)} \\
 && + \frac{1}{2} {\alpha_0}^{(i)} e_j \cdot \kappa \otimes A_{e_j}e_i + \frac{1}{2} {\alpha_6}^{(i)} Je_j \cdot \kappa \otimes A_{e_j}e_i, \\
- \frac{3}{4} Je_i \cdot Je_j \cdot \alpha^{(i)} \otimes e_j &=& \frac{3}{4} \kappa \otimes ({\alpha_0}^{(i)}e_i + {\alpha_6}^{(i)}Je_i + A_{e_i}{\alpha_1}^{(i)}) \\
 && + \frac{3}{4} \vol \cdot \kappa \otimes (-{\alpha_0}^{(i)}Je_i + {\alpha_6}^{(i)}e_i + JA_{e_i}{\alpha_1}^{(i)}) \\
 && + \{ \wedge^1 M \otimes \wedge^1 M {\rm \; part} \}, \\
 \frac{1}{2} (e_j \mathbin{\lrcorner} \psi^- \cdot \otimes \Id)\bar{\nabla}_{e_j}(\alpha^{(i)} \otimes e_i) &=& - \kappa \otimes \bar{\nabla}_{e_j}(g(e_j,{\alpha_1}^{(i)})e_i) \\
 && + g(\bar{\nabla}_{e_j}e_j, {\alpha_1}^{(i)}) \kappa \otimes e_i \\
 && - \vol \cdot \kappa \otimes \bar{\nabla}_{e_j}(g(Je_j,{\alpha_1}^{(i)})e_i) \\
 && + g(J\bar{\nabla}_{e_j}e_j, {\alpha_1}^{(i)}) \vol \cdot \kappa \otimes e_j \\
 && + \{ \wedge^1 M \otimes \wedge^1 M {\rm \; part} \}, \\
(\Id \otimes A_{e_j})\bar{\nabla}_{e_j}(\alpha^{(i)} \otimes e_i) &=& \kappa \otimes A_{e_j}\bar{\nabla}_{e_j}({\alpha_0}^{(i)}e_i) + \vol \cdot \kappa \otimes A_{e_j}\bar{\nabla}_{e_j}({\alpha_6}^{(i)}e_i) \\
 && + (\Id \otimes A_{e_j})\bar{\nabla}_{e_j}({\alpha_1}^{(i)} \cdot \kappa \otimes e_i).
\end{eqnarray*}
We project the equation $(\ref{D_{TM}^2=0 iff})$ onto $\wedge^0 M \otimes \wedge^1M$ and $\wedge^6 M \otimes \wedge^1 M$ again. Considering the projection onto $\wedge^0M \otimes \wedge^1M \cong \wedge^1M$, we get 
\begin{eqnarray*}
\bar{\Delta} \left( {\alpha_0}^{(i)}e_i \right) &=& - \frac{25}{4}{\alpha_0}^{(i)}e_i - \frac{3}{4}{\alpha_6}^{(i)}Je_i + \frac{5}{4}A_{e_i}{\alpha_1}^{(i)} \\
&& - \bar{\nabla}_{e_j} \left( g(e_j, {\alpha_1}^{(i)})e_i \right) + A_{e_j}\bar{\nabla}_{e_j}\left( {\alpha_0}^{(i)}e_i \right) \\
&\overset{(\ref{D_{TM}=0, wedge^0 otimes wedge^1})}{=}& -4{\alpha_0}^{(i)}e_i + A_{e_j}\bar{\nabla}_{e_j} \left( {\alpha_0}^{(i)}e_i \right).
\end{eqnarray*}
From the fact that $\bar{\Delta}-\Delta = A_{e_j}\bar{\nabla}_{e_j} - 2\Id$ over the tangent bundle $TM$, the above equation becomes
\begin{equation*}
\Delta \left( {\alpha_0}^{(i)}e_i \right) = -2{\alpha_0}^{(i)}e_i.
\end{equation*}
Since the Laplace operator on a tangent bundle is a non-negative operator, the equation ${\alpha_0}^{(i)}e_i=0$ follows from this equation. Similarly for $\wedge^6M \otimes \wedge^1M \cong \wedge^1M$, we know ${\alpha_6}^{(i)}e_i=0$.
\end{proof}

Applying the results of Lemma \ref{lemma;alpha_0,alpha_6}, ${\alpha_0}^{(i)}e_i=0$ and ${\alpha_6}^{(i)}e_i=0$, to the equation $(\ref{D_{TM}=0})$, we have
\begin{equation} \label{simply D_{TM}=0}
 \begin{split}
 0&= - \kappa \otimes \bar{\nabla}_{e_j}(g(e_j, {\alpha_1}^{(i)})e_i) - \vol \cdot \kappa \otimes \bar{\nabla}_{e_j}(g(e_j, J{\alpha_1}^{(i)})e_i) \\
 &\quad - \frac{1}{2} e_i \cdot \kappa \otimes {\alpha_1}^{(i)} - \frac{1}{2} Je_i \cdot \kappa \otimes J{\alpha_1}^{(i)} + (A_{e_j} \otimes \Id)\bar{\nabla}_{e_j}({\alpha_1}^{(i)} \cdot \kappa \otimes e_i).
 \end{split}
\end{equation}

As already mentioned in the proof of Lemma \ref{lemma;alpha_0,alpha_6}, we want to project the equation $(\ref{condition Rarita1})$ (briefly the equation $(\ref{simply D_{TM}=0})$) onto bundles other than $\wedge^1 M$ of the decomposition $(\ref{irrer decomp S_{1/2} otimes TM})$. For this purpose, we introduce several symbols:
\begin{eqnarray} 
w &\coloneqq& \pr_{\wedge^2 M} \left( {\alpha_1}^{(i)} \otimes e_i \right) = \frac{1}{2} {\alpha_1}^{(i)} \wedge e_i = \frac{1}{2} \left( {\alpha_1}^{(i)} \otimes e_i - e_i \otimes {\alpha_1}^{(i)} \right), \label{w} \\
H &\coloneqq& \pr_{\Sym M} \left( {\alpha_1}^{(i)} \otimes e_i \right) = \frac{1}{2} {\alpha_1}^{(i)} \odot e_i = \frac{1}{2} \left( {\alpha_1}^{(i)} \otimes e_i + e_i \otimes {\alpha_1}^{(i)} \right), \label{H} \\
h &\coloneqq& \pr_{\Sym^+ M} H = \frac{1}{4} \left( {\alpha_1}^{(i)} \odot e_i + J{\alpha_1}^{(i)} \odot Je_i \right), \label{h} \\
S &\coloneqq& \pr_{\Sym^- M} H = \frac{1}{4} \left( {\alpha_1}^{(i)} \odot e_i - J{\alpha_1}^{(i)} \odot Je_i \right), \label{S} \\
\varphi(\cdot, \; \cdot) &\coloneqq& g(Jh \cdot, \; \cdot) \in \Omega^{(1,1)} M, \label{varphi} \\
\sigma &\coloneqq& S_{\star}\psi^+ \in\Omega^{(2,1)+(1,2)}_0 M, \label{sigma}
\end{eqnarray}
where $\wedge^2 M$ denotes $\R\omega \oplus \wedge^{(2,0)+(0,2)}M \oplus \wedge^{(1,1)}_0M$.

\let\temp\thetheorem
\renewcommand{\thetheorem}{\ref{thmA}}

\begin{theorem}
Let $(M,g,J)$ be a 6-dimensional compact strict nearly K\"{a}hler manifold, then the space of the Rarita-Schwinger fields is isomorphic to the space of the harmonic 3-forms. In particular, we get
\begin{equation*}
\dim \ker Q = b_3(M),
\end{equation*}
where $b_3(M)$ denotes the 3rd Betti number.
\end{theorem}

\let\thetheorem\temp
\addtocounter{theorem}{-1}

\begin{proof}
Let $\alpha^{(i)} \otimes e_i$ be a local expression of a Rarita-Schwinger field. Employing Lemma \ref{lemma;alpha_0,alpha_6} and the equation $(\ref{S_{3/2} condition omega^1})$, we get $A_{e_i}{\alpha_1}^{(i)}=0$. This implies ${\alpha_1}^{(i)} \wedge e_i \in C^{\infty}(M)\omega \oplus \Omega^{(1,1)}_0 M$. Combining this with $(\ref{S_{3/2} condition omega^6})$, we obtain $w \in \Omega^{(1,1)}_0 M$. Furthermore, we know $h \in \Gamma(\Sym^+_0 M)$, and of course $\varphi \in \Omega^{(1,1)}_0 M$.

Next, the equation $(\ref{D_{TM}=0, wedge^0 otimes wedge^1})$ becomes $\bar{\nabla}_{e_j}(g(e_j, {\alpha_1}^{(i)})e_i)=0$ from Lemma \ref{lemma;alpha_0,alpha_6}, and this is equal to $\delta w + \delta h + \delta S=0$. In the same way, $\delta w + \delta h - \delta S = 0$ follows immediately from the equation $(\ref{D_{TM}=0, wedge^6 otimes wedge^1})$. Putting these two equations together, we obtain
\begin{equation} \label{deltaw,h,S}
\delta w + \delta h=0, \; \delta S=0.
\end{equation}

Now let us project the equation $(\ref{simply D_{TM}=0})$ onto the bundle $\wedge^1 M \otimes \wedge^1 M$. Then we get
\[
(A_{e_j} \otimes \Id)\bar{\nabla}_{e_j}({\alpha_1}^{(i)} \otimes e_i) = \frac{1}{2} e_i \otimes {\alpha_1}^{(i)} + \frac{1}{2} Je_i \otimes J{\alpha_1}^{(i)} = -w + h.
\]
Furthermore, using the two actions $A_{e_j \star}$ and $\widetilde{A_{e_j \star}}$ on 2-tensors introduced in Subsection \ref{action 2-tensors}, we find that the above equation becomes
\begin{equation} \label{D_{TM}=0, wedge^1 otimes wedge^1'}
\frac{1}{2} A_{e_j \star}\bar{\nabla}_{e_j}({\alpha_1}^{(i)} \otimes e_i) + \frac{1}{2} \widetilde{A_{e_j \star}}\bar{\nabla}_{e_j}({\alpha_1}^{(i)} \otimes e_i) = -w + h. 
\end{equation}

Note that ${\alpha_1}^{(i)} \otimes e_i$ is denoted as $w+h+S$. Applying the results of Lemma \ref{lem:herm} and Proposition \ref{prop; action 2-tensor} to the equation $(\ref{D_{TM}=0, wedge^1 otimes wedge^1'})$, we get
\begin{equation} \label{important eq}
 \begin{split}
 -2w + 2h &= \delta w \mathbin{\lrcorner} \psi^- + (J\delta S \mathbin{\lrcorner} \psi^- + \delta \sigma) \circ J + A_{e_j \star} \bar{\nabla}_{e_j}h \\
 &\quad - \delta h \mathbin{\lrcorner} \psi^- + \ast d\sigma - \delta S \mathbin{\lrcorner} \psi^- + \widetilde{A_{e_j \star}}\bar{\nabla}_{e_j}w.
 \end{split}
\end{equation}
Therefore, the equations $(\ref{deltaw,h,S})$ and $(\ref{important eq})$ are equivalent to the system
\begin{equation} \label{a}
\begin{cases}
\ast d\sigma = -2w, & \\
\delta \sigma \circ J = 2h, & \\
A_{e_j \star}\bar{\nabla}_{e_j \star}h + \widetilde{A_{e_j \star}}\bar{\nabla}_{e_j}w=0, & \\
(\delta w - \delta h - \delta S) \mathbin{\lrcorner} \psi^- + (J\delta S \mathbin{\lrcorner} \psi^-) \circ J =0, & \\
\delta w + \delta h=0, \; \delta S=0.
\end{cases}
\end{equation}
By combining with $J$, the second equation of $(\ref{a})$ becomes
\begin{equation} \label{a second eq}
- \delta \sigma = 2\varphi.
\end{equation}
Let the third equation of $(\ref{a})$ act on $\psi^+$. Employing Lemma \ref{lem:herm} and Proposition \ref{prop; action 2-tensor} again, we find
\begin{equation}
2(\delta h + \delta w)\wedge \omega -4d\varphi + 4\ast dw=0.
\end{equation}
For the fourth equation of $(\ref{a})$, taking note of $(J\delta S \mathbin{\lrcorner} \psi^-) \circ J = - \delta S \mathbin{\lrcorner} \psi^-$, we get
\begin{equation}
\delta w - \delta h - 2\delta S = 0.
\end{equation}
Putting this equation together with the fifth equation of $(\ref{a})$, we get
\begin{equation} \label{a fifth eq}
\delta w=0, \; \delta h=0, \; {\rm and} \; \delta S=0.
\end{equation}
Based on these calculations $(\ref{a second eq})$-$(\ref{a fifth eq})$, the system $(\ref{a})$ is equivalent to
\begin{equation} \label{b}
\begin{cases}
\ast d\sigma = -2w, & \\
\delta \sigma = -2\varphi, & \\
\ast dw = d\varphi, & \\
\delta w = 0, \; \delta \varphi=0.
\end{cases}
\end{equation}
Here, we use the relation $\delta h = -J\delta\varphi$. It is a little difficult to derive the equation $\delta S = 0$ when we show $(\ref{a})$ from $(\ref{b})$. We solve this problem by checking in which component of the decomposition $(\ref{irrer decomp S_{1/2} otimes TM})$ each term of the formula $(\ref{second induced action Sym^-})$ is contained. Finally, combining the equations in this system $(\ref{b})$, we obtain the system equivalent to $(\ref{b})$:
\begin{equation}
\Delta \sigma = 0, \; \varphi=0, \; w=0.
\end{equation}
Here, we note that $M$ is compact.

Conversely, if there exists a harmonic primitive $(2,1)+(1,2)$-form $\sigma$, then we have a Rarita-Schwinger field $\alpha^{(i)}\otimes e_i$ immediately.  By Verbitsky's theorem (cf.\cite{Verbitsky}), arbitrary harmonic 3-form is primitive and of type $(2,1)+(1,2)$, so Theorem \ref{thmA} is proved.
\end{proof}

\section{Examples and Applications} \label{Examples and Applications}
\subsection{Examples}
As shown in Theorem \ref{thmA}, we revealed that the space of the Rarita-Schwinger fields isomorphic to the space of the harmonic 3-form on 6-dimensional compact strict nearly K\"{a}hler manifolds. It is an important problem whether there actually exist manifolds with non-zero Rarita-Schwinger fields. We find out how many Rarita-Schwinger fields exist for a specific manifold.

There are only six known examples of complete simply connected strict nearly K\"{a}hler 6-manifolds. Four of them are homogeneous: $S^6 = \G_2 / \SU(3)$, $S^3 \times S^3 = (\SU(2) \times \SU(2) \times \SU(2)) / \Delta \SU(2)$, $\C\P^3 = \Sp(2) / (\U(1) \times \Sp(1))$, $\F(1,2) = \SU(3) / T^2$. They were first constructed by J. A. Wolf and A. Gray \cite{WolfGray}. Moreover, Butruille \cite{Butruille} showed that there is no other homogeneous example in dimension 6. The other two are inhomogeneous examples, $S^6$ and $S^3 \times S^3$, found by Foscolo and Haskins \cite{FoscoloHaskins}.

We actually calculate the 3rd betti numbers for these six nearly K\"{a}hler manifolds, we get $b_3(S^3 \times S^3)=2$ (both for homogeneous and inhomogeneous) and the other cases vanish. Therefore we know that Rarita-Schwinger fields exist in two $S^3 \times S^3$ and not in the other.

We remark that the manifold $S^3 \times S^3$ with the standard metric, which is not a nearly K\"{a}hler manifold, does not have Rarita-Schwinger fields. This is easily indicated by the following.

Let $\{e_1, e_2, e_3 \}$ be a local orthonormal tangent frame of the first factor of the Riemannian product $S^3 \times S^3$, and $\{e_4, e_5, e_6 \}$ be a local orthonormal tangent frame of the second factor. Any sections of $S_{1/2} \otimes T(S^3 \times S^3)$ are written locally as $\sum_i \alpha^{(i)} \otimes e_i$. Using the Lichnerowicz formula for the twisted Dirac operator, we get
\begin{equation*}
 \begin{split}
 \sum_i {D_{TM}}^2 (\alpha^{(i)} \otimes e_i) &= \sum_i \nabla^{\ast}\nabla(\alpha^{(i)} \otimes e_i) + \frac{\scal}{4} \sum_i \alpha^{(i)} \otimes e_i \\
 &\quad + \frac{1}{2} \sum_{i,j,k} e_j e_k \alpha^{(i)} \otimes R(e_j, e_k)e_i \\
 &= \sum_i \nabla^{\ast}\nabla(\alpha^{(i)} \otimes e_i) + 4 \sum_i \alpha^{(i)} \otimes e_i \\
 &\quad + \sum_{1 \leq i,j \leq 3} e_j e_i \alpha^{(i)} \otimes e_j + \sum_{4 \leq i,j \leq 6} e_j e_i \alpha^{(i)} \otimes e_j.
 \end{split}
\end{equation*}
By taking the $L^2$-inner product with $\sum_i \alpha^{(i)} \otimes e_i$, the above equation becomes
\begin{equation} \label{S^3 times S^3}
 \begin{split}
 \left| {D_{TM}} \left(\sum_i \alpha^{(i)} \otimes e_i \right) \right|^2 &= \left| \nabla \left(\sum_i \alpha^{(i)} \otimes e_i \right) \right|^2 + 4\left| \sum_i \alpha^{(i)} \otimes e_i \right|^2 \\
 &\quad - \left| \sum_{1 \leq i \leq 3} e_i \cdot \alpha^{(i)} \right|^2 - \left| \sum_{4 \leq i \leq 6} e_i \cdot \alpha^{(i)} \right|^2 .
 \end{split}
\end{equation}
Evaluating the third term of the equation $(\ref{S^3 times S^3})$, we have
\begin{equation*}
 \begin{split}
 \left| \sum_{1 \leq i \leq 3} e_i \cdot \alpha^{(i)} \right|^2 &\leq \left( \sum_{1 \leq i \leq 3} \left| e_i \cdot \alpha^{(i)} \right| \right)^2 = \left( \sum_{1 \leq i \leq 3} \left| \alpha^{(i)} \right| \right)^2 \\
 &= \sum_{1 \leq i \leq 3} \left| \alpha^{(i)} \right|^2 + 2 \left| \alpha^{(1)} \right| \left| \alpha^{(2)} \right| + 2 \left| \alpha^{(2)} \right| \left| \alpha^{(3)} \right| + 2 \left| \alpha^{(3)} \right| \left| \alpha^{(1)} \right| \\
 &\leq 3 \sum_{1 \leq i \leq 3} \left| \alpha^{(i)} \right|^2.
 \end{split}
\end{equation*}
We evaluate the fourth term in the same way. Putting these calculations together, we get the inequality
\begin{equation*}
\left| {D_{TM}} \left(\sum_i \alpha^{(i)} \otimes e_i \right) \right|^2 \geq \left| \nabla \left(\sum_i \alpha^{(i)} \otimes e_i \right) \right|^2 + \left| \sum_i \alpha^{(i)} \otimes e_i \right|^2.
\end{equation*}
From this equation, we deduce $\ker D_{TM} = \{ 0 \}$. In other words, we know that there is no Rarita-Schwinger field on $S^3 \times S^3$ with the standard metric.

This is the first example in which the existence of non-trivial Rarita-Schwinger fields depends not only on topological conditions but also on metrics. In both of the nearly K\"{a}hler and the standard $S^3 \times S^3$, there is no harmonic spinor field. However, only in the nearly K\"{a}hler $S^3 \times S^3$, there are Rarita-Schwinger fields.

\subsection{Linear stability}
The following theorem about the linear stability of nearly K\"{a}hler manifolds is known by Semmelmann, Wang, and Wang \cite{SemmelmannWang}.

\begin{theorem}[\cite{SemmelmannWang}]
Let $(M, g, J)$ be a complete strict nearly K\"{a}hler 6-manifold. If $b_2(M)$ or $b_3(M)$
is nonzero, then $g$ is linearly unstable with respect to the Einstein-Hilbert action restricted to
the space of Riemannian metrics with constant scalar curvature and fixed volume.
\end{theorem}

Since $\dim \ker Q=b_3(M)$ is obtained in our main result (Theorem \ref{thmA}), we know that $g$ is linearly unstable in the above sense on six-dimensional complete strict nearly K\"{a}hler manifolds with non-zero Rarita-Schwinger fields.

\subsection{Infinitesimal deformation of Killing spinors}
For detailed information about Killing spinor variations, please refer to \cite{Wang} and other sources. 

We define an infinitesimal deformation of Killing spinors for a general $n$-dimensional Riemannian manifold $(M,g)$. Before that, we provide a tensor $\Psi^{(\beta, \kappa_0)}$ in $\Gamma(T^{\ast}M \otimes S_{1/2})$ for $\beta : TM \rightarrow TM$ and $\kappa_0$ in $\Gamma(S_{1/2})$:
\[
\Psi^{(\beta, \kappa_0)} (X) = \beta(X) \cdot \kappa_0.
\]

\begin{definition} \label{def: infinitesimal deformation of the Killing spinor} 
We call a pair $(\beta, \kappa)$ an infinitesimal deformation of the Killing spinor $\kappa_0$ with constant $c$ if $\beta: TM \rightarrow TM$ symmetric and $\kappa$ in $\Gamma(S_{1/2})$ satisfy
 \begin{enumerate}
 \renewcommand{\labelenumi}{(\roman{enumi})}
 \item $\kappa$ is a Killing spinor with constant $c$. 
 \item $\tr \beta = \delta\beta=0$.
 \item $D_{TM} \Psi^{(\beta, \kappa_0)} = nc\Psi^{(\beta, \kappa_0)}$.
 \end{enumerate}
\end{definition}

We now consider a compact 6-dimensional strict nearly K\"{a}hler manifold $(M^6, g, J)$ with scalar curvature $\scal = 30$ except the round sphere $S^6$. In this case, of course, $n=6$ and $c=\frac{1}{2}$ in Definition \ref{def: infinitesimal deformation of the Killing spinor}. Furthermore, there is a one-to-one correspondence between a Killing spinor (modulo constant rescaling) and a nearly K\"{a}hler structure (cf. \cite{Grunewald}). So, the Killing spinor $\kappa_0$ is corresponding to the nearly K\"{a}hler structure $(g,J)$. For (i) in the above definition, we take $\kappa$ as a constant multiplication of $\kappa_0$ because the space of the Killing spinors with the Killing number $\frac{1}{2}$ is one-dimensional. We adjust the symbols to the ones we used in Section \ref{Rarita Schwinger fields on nearly Kahler}. First, since $\beta$ is an endomorphism of $TM$, it is expressed locally as ${\alpha_1}^{(i)} \otimes e_i$. Furthermore, since $\beta$ is symmetric, it becomes $H$ defined in $(\ref{H})$. This means that $w$, defined in $(\ref{w})$, is zero. Therefore, (ii) in the above definition is rewritten as
\[
H \in \Sym_0 M, \; \delta H=0.
\]
The condition (iii) becomes
\begin{equation} \label{D_{TM}=3}
D_{TM} ({\alpha_1}^{(i)} \cdot \kappa_0 \otimes e_i) = 3{\alpha_1}^{(i)} \cdot \kappa_0 \otimes e_i.
\end{equation}
Projecting the left-hand side of the equation $(\ref{D_{TM}=3})$ onto the bundle $\wedge^1 M \otimes \wedge^1 M$, we get
\[
(A_{e_j} \otimes \Id)\bar{\nabla}_{e_j}({\alpha_1}^{(i)} \otimes e_i) - h.
\]
Projecting the right-hand side of the equation $(\ref{D_{TM}=3})$, we get 
\[
3h + 3S.
\]
Putting the above equations together, we have
\[
(A_{e_j} \otimes \Id)\bar{\nabla}_{e_j}({\alpha_1}^{(i)} \otimes e_i) =  4h + 3S.
\]
Continuing with the calculation similar to the one in Section \ref{Rarita Schwinger fields on nearly Kahler}, we find that
\[
\delta\sigma = -8\varphi, \; d\varphi = -\frac{3}{2}\sigma, \; \varphi \in \Omega^{(1,1)}_0 M, \; \sigma \in \Omega^{(2,1)+(1,2)}_0 M.
\]
This is equivalent to
\[
\Delta \varphi = 12 \varphi, \; \delta \varphi = 0, \; \sigma = - \frac{2}{3}d\varphi, \; \varphi \in \Omega^{(1,1)}_0 M,
\]
so we obtain the following result. Note that the differential of a co-closed promitive $(1,1)$-form is contained in the space of primitive $(2,1)+(1,2)$-forms. In the case of the round sphere $S^6$, the same arguments hold and the dimension of the space of the Killing spinors is eight.

\let\temp\thetheorem
\renewcommand{\thetheorem}{\ref{thmB}}

\begin{theorem}
Let $(M^6, g, J)$ be a compact strict nearly K\"{a}hler 6-manifold. Then the space of the infinitesimal deformations of Killing spinors is isomorphic to the direct sum of the space of the primitive co-closed $(1,1)$-eigenforms of the Laplace operator for the eigenvalue 12 and the space of the Killing spinors with constant $\frac{1}{2}$.
\end{theorem}

\let\thetheorem\temp
\addtocounter{theorem}{-1}

We denote by $E(\lambda)$ the $\lambda$-eigenspace of $\Delta$ restricted to the space of primitive co-closed $(1,1)$-forms. Besides the round sphere $S^6$, Moroianu, Nagy, and Semmelmann \cite[Theorem 4.1]{MoroianuNagySemmelmann} showed that the space of the infinitesimal deformations of nearly K\"{a}hler structures is $E(12)$, which is a part of the space of the essential infinitesimal Einstein deformations $E(2) \oplus E(6) \oplus E(12)$ \cite[Lemma 5.2]{MoroianuSemmelmann}. As described in the paper \cite{Grunewald}, the space of the Killing spinors (modulo constant rescaling) is one-to-one correspondence with the space of the nearly K\"{a}hler structures except $S^6$. So, the space of the infinitesimal deformations of Killing spinors (up to a constant) and the infinitesimal deformations of nearly K\"{a}hler structures are consistent. Thus, Theorem \ref{thmB} is a reproof of the result of \cite{MoroianuNagySemmelmann} through the Killing spinor.

\section*{Acknowledgement}
We wish to offer our immeasurable gratitude to Professor Yasushi Homma for useful discussions, advices, and providing us with the topic. This paper is a part of the outcome of research performed under a Waseda University Grant for Special Research Projects (Project number: 2020C-614).


\begin{thebibliography}{99}
 \bibitem{BarMazzeo} C. B\"{a}r, R. Mazzeo. {\it Manifolds with Many Rarita-Schwinger Fields}. Comm. Math. Phys. 384 (2021), no. 1, 533-548. 
 \bibitem{Butruille} J.-B. Butruille. {\it Classification des vari\'{e}t\'{e}s approximativement k\"{a}hleriennes homog\`{e}nes}. Ann. Global Anal. Geom. 27 (2005), no. 3, 201-225.
 \bibitem{CharbonneauHarland} B. Charbonneau, D. Harland. {\it Deformations of nearly K\"{a}hler instantons}. Comm. Math. Phys. 348 (2016), no. 3, 959-990.
 \bibitem{FoscoloHaskins} L. Foscolo, M. Haskins. {\it New $\G_2$-holonomy cones and exotic nearly K\"{a}hler structures on $S^6$ and $S^3 \times S^3$}. Ann. of Math. (2) 185 (2017), no. 1, 59-130.
 \bibitem{Foscolo} L. Foscolo. {\it Deformation theory of nearly K\"{a}hler manifolds}. J. Lond. Math. Soc. (2) 95 (2017), no.2, 586-612.
 \bibitem{Gray2} A. Gray. {\it Nearly K\"{a}hler manifolds}. J. Differential Geometry 4 (1970), 283–309.
 \bibitem{Gray} A. Gray. {\it The structure of nearly K\"{a}hler manifolds}. Math. Ann. 223 (1976), no. 3, 233-248.
 \bibitem{Grunewald} R. Grunewald. {\it Six-dimensional Riemannian manifolds with a real Killing spinor}. Ann. Global Anal. Geom. 8 (1990), no.1, 43-59. 
 \bibitem{Hitchin} N. Hitchin. {\it Stable forms and special metrics}. Global differential geometry: the mathematical legacy of Alfred Gray (Bilbao, 2000), 70–89, Contemp. Math., 288, Amer. Math. Soc., Providence, RI, 2001.
 \bibitem{Yasushi} Y. Homma. {\it Twisted Dirac operators and generalized gradients}. Ann. Global Anal. Geom. 50 (2016), no. 2, 101-127.
 \bibitem{YasushiSemmelmann} Y. Homma, U. Semmelmann. {\it The kernel of Rarita-Schwinger operator on Riemannian spin manifolds}. Comm. Math. Phys. 370 (2019), no. 3, 853-871.
 \bibitem{YasushiTomihisa} Y. Homma, T. Tomihisa. {\it Spectra of the Rarita-Schwinger operator on some symmetric spaces}. J. Lie Theory 31 (2021), no. 1, 249-264.
 \bibitem{MoroianuNagySemmelmann} A. Moroianu, P.-A. Nagy, U. Semmelmann. {\it Deformations of nearly K\"{a}hler structures}. Pacific J. Math. 235 (2008), no. 1, 57-72.
 \bibitem{MoroianuSemmelmann2} A. Moroianu, U. Semmelmann. {\it The Hermitian Laplace operator on nearly K\"{a}hler manifolds}. Comm. Math. Phys. 294 (2010), no. 1, 251-272.
 \bibitem{MoroianuSemmelmann} A. Moroianu, U. Semmelmann. {\it Infinitesimal Einstein deformations of nearly K\"{a}hler metrics}. Trans. Amer. Math. Soc. 363 (2011), no. 6, 3057-3069.
 \bibitem{NagyAndy} P.-A. Nagy. {\it Nearly K\"{a}hler geometry and Riemannian foliations}. Asian J. Math. 6 (2002), no. 3, 481-504.
 \bibitem{RaritaSchwinger} W. Rarita, J. Schwinger. {\it On a theory of particles with half-integral spin}. Phys. Rev. (2) 60, (1941), 61.
 \bibitem{SemmelmannWeingart} U. Semmelmann, G. Weingart {\it The standard Laplace operator}. Manuscripta Math. 158 (2019), no. 1-2, 273-293. 
 \bibitem{SemmelmannWang} U. Semmelmann, C. Wang, M. Y. Wang. {\it On the linear stability of nearly K\"{a}hler 6-manifolds}. Ann. Global Anal. Geom. 57 (2020), no. 1, 15-22.
 \bibitem{Coevering} C. van Coevering. {\it Deformations of Killing spinors on Sasakian and 3-Sasakian manifolds}. J. Math. Soc. Japan 69 (2017), no. 1, 53-91.
 \bibitem{Verbitsky} M. Verbitsky. {\it Hodge theory on nearly K\"{a}hler manifolds}. Geom. Topol. 15 (2011), no. 4, 2111-2133.
 \bibitem{Wang} M. Y. Wang. {\it Preserving parallel spinors under metric deformations.} Indiana Univ. Math. J. 40 (1991), no. 3, 815-844. 
 \bibitem{Witten} E. Witten. {\it Elliptic genera and quantum field theory}. Comm. Math. Phys. 109 (1987), no. 4, 525-536.
 \bibitem{WolfGray} J. A. Wolf, A. Gray. {\it Homogeneous spaces defined by Lie group automorphisms. II.} J. Differential Geometry 2 (1968), 115-159.
\end{thebibliography}
\end{document}